\documentclass[11pt,a4paper]{article}
\usepackage{amsmath,amsthm,amsfonts,latexsym,amscd,amssymb, makeidx,graphics, caption}
\makeindex
\usepackage{float}
\usepackage{bm}
\usepackage{enumitem}
\usepackage{tabularx}
\usepackage{breqn}
\usepackage{microtype}
\newtheorem{theorem}{Theorem}[section]
\newtheorem{lemma}[theorem]{Lemma}

\theoremstyle{definition}

\theoremstyle{remark}
\newtheorem{remark}{Remark}

\usepackage{mathtools}

\usepackage[table]{xcolor}
\usepackage{mathptmx}
\usepackage{lineno}

\def\addsymbol #1: #2#3{$#1$ \> \parbox{5in}{#2 \dotfill \pageref{#3}}\\}

\begin{document}

	\begin{center}
		\huge
		{On the number of orthomorphisms of the alternating group on four symbols} 
	\end{center}
		\begin{center}
	  Vivek Kumar Jain and	Rohitesh Pradhan\\
	Department of Mathematics, Central University of South Bihar, Gaya, India\\
		Email:~jaijinenedra@gmail.com, rohiteshpradhan@gmail.com
	\end{center}

	\large{\textbf{Abstract:}}
In this study, the set of orthomorphisms of a finite group $G$ of order $3M$ with a normal subgroup of order $M$ has been partitioned into classes. Additionally, a number of relationships between the orders of these classes have been established. Ultimately, this approach has been used to theoretically determine all 3840 orthomorphisms of alternating groups on four symbols by just counting 30 orthomorphisms—a problem that had been unsolved since 1992.

	\vspace{.5cm}
	
		\smallskip
	
	
	\section{Introduction}\label{sec1}
	Let $G$ be a group. A bijection $\theta$: $G \rightarrow G$, for which the map $\phi_{\theta}$: $x \mapsto x^{-1}\theta(x)$ is also a bijection from $G$ to itself, is called an \emph{orthomorphism} of $G$ and the map $\phi_{\theta}$ is known as \emph{complete mapping}. An orthomorphism of a group which fixes identity element of the group is called \emph{normalized} orthomorphism. Now onwards, by an orthomorphism, we mean a normalized orthomorphism.  Two orthomorphisms $\theta_{1}$ and $\theta_{2}$ of $G$ are called \emph{orthogonal}, written  $\theta_{1} \perp \theta_{2}$, if the map  $x \mapsto \theta_{1}(x)^{-1}\theta_{2}(x)$ is a bijection from $G$ to itself. Equivalently, $\theta_{1} \perp \theta_{2}$ if and only if $\theta_{1}\theta_{2}^{-1}$ is  an orthomorphism of $G$ \cite[Corollary 1.31, p. 22]{evans}. We denote the set of orthomorphisms of a group $G$ by Orth($G$). We know  that the Cayley table of $G$ is  a Latin square, and an orthomorphism of $G$ gives us a Latin square which is orthogonal to the Cayley table of $G$. Each transversal of the Cayley table of G gives an orthomorphism of $G$ \cite{wanless}.
	
	For groups of order less than equal to $23$ the number of orthomorphisms is known mainly due to computer searches: see Table 13.1 in \cite{evans}, Table 1 in \cite{mckay}, or Table 4 in \cite{wanless}.
	The Hall-Paige theorem \cite{hall} states that groups having non-trivial, cyclic Sylow 2-subgroups have no orthomorphisms. Groups of order less than equal to five are easily handled by hand computation. Euler \cite{euler} showed that $\mathbb{Z}_{7}$ has $19$ normalized orthomorphisms by listing the transversals of the cyclic Latin square of order seven. There are $48$  orthomorphisms for each of the noncyclic groups of order eight proven theoretically by Bedford and Whitaker \cite{bedford}. The orthomorphisms of elementary abelian group of order $8$ can be found either by using permutation polynomials as given in Section 9.2 of \cite{evans}  or by solving systems of linear equations, as given in Section 10.3 and 10.4 of \cite{evans}. If $\mathbb{Z}_{n}$ denotes the cyclic group of order $n$, then the number of orthomorphisms of $\mathbb{Z}_{2} \times \mathbb{Z}_{4}$ and its cycle structure is given in \cite{jain}.  	
	 In 1964, Chang and Tai \cite{chang}, through exhaustive hand computations found that alternating group on four symbols$(A_{4})$ has $3776$ normalized orthomorphisms. In 1965, Hall and Knuth \cite{knuth} using computer found that $A_{4}$ admits $3840$ normalized complete mappings. Later, in $2000$ and $2011$, via computer searches, Shieh et al. \cite{shieh} and Wanless \cite{wanless} confirmed that $A_{4}$ has $3840$ normalized complete mapping respectively.

	  In 1992, Evans  in \cite[Problem 22, p. 98]{evans3}  asked the following problem.
	 
	 \noindent \textbf{Problem:} Give a theoretical proof that $|$Orth($A_{4}$)$| = 3840$.\\
	 Again in 2018 \cite[Problem 16.76, p. 514]{evans}, Evans asked the same problem. We have solved this long standing problem in this paper.\\
 For groups of order $12$ there are three groups namely $A_{4},$ the dihedral group of order $12$ $(D_{12})$ and $\mathbb{Z}_{2} \times \mathbb{Z}_{6}$, which has orthomorphisms. The group $D_{12}$ and $\mathbb{Z}_{2} \times \mathbb{Z}_{6}$ have $6336$ and $16512$ normalized orthomorphisms respectively see Table 13.1 in \cite{evans}. In 2023, Evans \cite{evans2}  gives a method to determine the number of orthomorphisms for an abelian group $G$ of order $9$ with a index $3$ normal subgroup and also indicated how his method can be applied to other groups.
 
  In this paper, we generalized this method to non-abelian groups with a normal subgroup of index 3.
 Establishment of new relations (Lemma 2.3, Lemma 2.4 and 2.5) which significantly reduces the computational effort required to determine the number of orthomorphisms in finite groups possessing a normal subgroup of index $3$. While earlier studies relied on direct enumeration, our approach provides a general framework that reduces the computation.  Alternating group ($A_4$) is the smallest non abelian group with a index $3$ normal subgroup. Orth($A_4$) has been divided into $10$ disjoint classes. These classes are same as introduced in [6]. The new relations of lemmas mentioned above reduces $10$ classes into $4$ classes. These lemmas are also applicable to any finite groups having a normal subgroup of index $3$.  Each reduced class has been further divided into sub-classes. Then, by applying automorphism of Orth($A_4$) on sub-classes, we showed that the number of elements in some sub-classes are same. Ultimately, by determining only $30$ orthomorphisms through hand computation, we counted all 3840 orthomorphisms of $A_4$. 
   \\ \\
	\textbf{Notations} \\
	Let $G$ be a group. We fix following notations for the sake of convenience:
	\begin{enumerate}
		\item [(i)] $S_{n} = Sym\{1,2,\ldots,n\}$,
		\item [(ii)] $A_{n} = Alt\{1,2,\ldots,n\}$,
		\item [(iii)] $\mathbb{Z}_{n}$ represents cyclic group of order $n$,
		\item [(iv)]  $D_{n}$ represents dihedral group of order $2n$,
		\item [(v)] $Aut(G)$ denotes the group of automorphism of $G$, 
		\item [(vi)] $i_{g}(x) = gxg^{-1}$, where $i_{g} \in Aut(G)$ and $x \in G$.
	\end{enumerate}
	 We take alternating group on four symbols as $A_{4}$ and symmetric group on four symbols as $S_{4}$.

	\section{Basic Results}
	In this section, we recall some known definitions for the sake of completeness.

		An \emph{automorphism} $A$ of Orth($G$) is a bijection from Orth($G$) to itself such that $A(\theta_{1}) \perp A(\theta_{2})$ if and only if $\theta_{1} \perp \theta_{2}$, where $\theta_{1}, \theta_{2} \in$ Orth($G$). Automorphisms of Orth($G$) are explained in detail in Section 8.1.3 in \cite{evans}. By \cite[Theorem 8.6, p. 206]{evans} following are examples of automorphisms of Orth($G$).
	\begin{enumerate}
		\item [(i)] For $\psi \in$ Aut($G$), the \textit{homology} $H_{\psi}$ is defined by $H_{\psi}(\theta) = \psi\theta \psi^{-1}$ .
		\item[(ii)] For $g \in G$, the \textit{translation} $T_{g}$ is defined by $T_{g}[\theta](x) = \theta(xg)\theta(g)^{-1}$.
		\item[(iii)] The \textit{reflection} $R$ is defined by $R[\theta](x) = x \theta(x^{-1})$.
	\end{enumerate}
	In this section, we will generalize the concept of collapsed permutation matrix given by Evans \cite{evans2}, from abelian group to any group.
	Let $G=\{g_1,g_2,\ldots,g_n\}$ be a group. Any bijection $\theta$ of $G$ can be represented by a permutation matrix $M_{\theta} = (m_{ij})_{|G| \times |G|}$, where 
	$$ m_{ij} =
	\begin{dcases}
		1 & if~\theta(g_j) = g_i,\\
		0 & otherwise.
	\end{dcases}$$
	The \emph{$g_k$th diagonal} of $M_{\theta}$ is the set of cells $\{(i,j) : g_{j}^{-1}g_{i} = g_k\}$.
	
	\begin{lemma}\label{le1}
		Let $\theta$ be a permutation of a group $G=\{g_1,g_2,\ldots,g_n\}$. Then $\theta$ is an orthomorphism of $G$ if and only if each row, column, and diagonal of $M_{\theta}$ contains exactly one 1. 
	\end{lemma}
	\begin{proof}
		Suppose $\theta$ is an orthomorphism of $G$. As $\theta$ is bijection, so each row and column of $M_{\theta}$ contains exactly one $1$. Also, $\phi_{\theta}$ is a bijection, which gives us that every $g_{k}{th}$ diagonal of $M_{\theta}$ contains exactly one $1$. Converse is straight forward.
	\end{proof}
	\subsection{Collapsed permutation matrix}
	Let $G$ be a group with identity element $e$, $|G| = ms$, $H=\{h_{1} = e,h_{2},\ldots,h_{m}\}$ be its normal subgroup of index $s$, and let $D = \{d_{1}=e,d_{2}, \ldots, d_{s}\}$ be an ordered set of distinct coset representatives for $H$ in $G$. Then  $G = d_{1}H \sqcup \ldots \sqcup d_{s}H$ can be arranged as an ordered set $$G=\{d_{1}h_{1},\ldots,d_{1}h_{m};d_{2}h_{1},\ldots,d_{2}h_{m};\ldots;d_{s}h_{1},\ldots,d_{s}h_{m}\}.$$ Let $\theta$ be a permutation of $G$. Consider a submatrix $A_{ij}$ of $M_{\theta}$ formed  by taking the rows $d_{i}h_{1},\ldots, d_{i}h_{m}$ and columns $d_{j}h_{1},\ldots, d_{j}h_{m}$. That is, $M_{\theta} = (A_{ij})_{s \times s}$.
	\[M_\theta = \begin{pmatrix}
		A_{11}&A_{12}&\ldots&A_{1s}\\
		A_{21}&A_{22}&\ldots&A_{2s}\\
		\vdots&\vdots&\vdots&\vdots\\
		A_{s1}&A_{s2}&\ldots&A_{ss}
	\end{pmatrix}.\]
	 If $a_{ij}$ denotes the number of $1$s in $A_{ij}$, then  the \emph{collapsed permutation matrix} corresponding to $\theta$ is $CM_{\theta}$, which is defined as $CM_{\theta}:= (a_{ij})_{s \times s}.$\\[.2cm]
	 \textbf{Note:} The above definitions are not new and are included here for the sake of completeness.
	\begin{lemma}\label{le2}
		Suppose $H$ is a normal subgroup of a group $G$ of index 3, $\vert H \vert$ = m, $D$ is a system of distinct coset representatives for $H$ in $G$, and $\theta$ is an orthomorphism of G. Then the collapsed permutation matrix of $\theta$ with respect to $H$ (and $D$) is
		$$ CM_{\theta} = \begin{pmatrix}
			\alpha&\beta&\gamma\\\beta&\gamma&\alpha\\\gamma&\alpha&\beta
		\end{pmatrix}$$
		for some $\alpha, \beta, \gamma$ such that $\alpha+\beta+\gamma = m$.
	\end{lemma}
	\begin{proof}
		Same as proof of Lemma 2.2 in \cite{evans2}.
	\end{proof}
	\begin{remark}\label{re1}
		If the collapsed permutation matrix of $\theta$ with respect to $H$ (and $D$) is 
		$$ CM_{\theta} = \begin{pmatrix}
		\alpha&\beta&\gamma\\\beta&\gamma&\alpha\\\gamma&\alpha&\beta
	\end{pmatrix},$$
		then we will say that $\theta$ is in the $class~\alpha \beta \gamma$. Clearly,  $class~\alpha \beta \gamma = \{\theta \in \text{Orth}(G): |\theta(H^{\ast}) \cap H^{\ast}|= \alpha - 1,|\theta(H^{\ast}) \cap aH| = \beta, |\theta(H^{\ast}) \cap bH|= \gamma\}$,  where $H^{\ast} = H \setminus \{e\}.$ The cardinality of   $class~\alpha \beta \gamma$  is given by $|\alpha \beta \gamma|$.
	\end{remark}
	\begin{remark}\label{re2}
		If $\theta \in $ Orth($G$), then $\theta^{-1} \in$ Orth($G$) by \cite[Theorem 8.7, p. 207]{evans}. It is easy to observe that for $\theta \in$ Orth($G$), $(M_{\theta})^{t} = M_{\theta^{-1}}$. Thus, $\theta$ and $\theta^{-1}$ belongs to the same class.
	\end{remark}
	\begin{remark}\label{re3}
		For $\theta \in$ Orth($G$), define a collection $\Delta_{ij} = \{\delta \in d_{i}H: \theta(\delta) \in d_{j}H\}.$ Clearly $|\Delta_{ij}| = a_{ji} = a_{ij}$ (By Lemma \ref{le2}). Also, if $\theta \in class~\alpha\beta\gamma$, then $|\Delta_{11}| = |\Delta_{23}| = |\Delta_{32}| = \alpha, |\Delta_{12}| = |\Delta_{21}| = |\Delta_{33}| = \beta$ and $|\Delta_{13}| = |\Delta_{22}| = |\Delta_{31}| = \gamma$. Also, note that $ \displaystyle	 \bigsqcup_{j = 1}^{3}\Delta_{ij} = d_{i}H $ and $\displaystyle \bigsqcup_{i = 1}^{3}\theta(\Delta_{ij}) = d_{j}H$.  
	
	\end{remark}
	\begin{lemma}\label{le3}
		Let $G$ be a group and $H$ its normal subgroup of index 3. Suppose $\theta \in Orth(G)$ and $\psi \in Aut(G)$ such that $\psi(H) = H, \psi(aH) = bH$ and $\psi(bH) = aH$, where $G = H \sqcup aH \sqcup bH.$ Then  $\bm{|\alpha \beta \gamma| = |\alpha \gamma \beta|}$.
	\end{lemma}
	\begin{proof}
		Let $G$ be a group and $H$ its normal subgroup of index $3$. Suppose $\theta \in \text{Orth}(G)$ and $\psi(h) = h'$ for some $h,h' \in H^{\ast}$, where $\psi \in Aut(G)$ such that $\psi(H) = H, \psi(aH) = bH$ and $\psi(bH) = aH$, where $G = H \sqcup aH \sqcup bH.$ Now, $(H_{\psi}(\theta))(h')$ = $\psi \theta \psi^{-1}(h')$ = $\psi (\theta(h))$. As a result, we get $(H_{\psi}(\theta))(H^{\ast}) = \psi(\theta(H^{\ast}))$. 		
		 Assume $\theta \in class~\alpha \beta \gamma$. Then, 
		 \begin{align*}
		 	|(H_{\psi}(\theta))(H^{\ast})\cap H^{\ast}|
		 	& =|\psi(\theta(H^{\ast})) \cap H^{\ast}|\\
		 	  &=|\psi(\theta(H^{\ast})\cap \psi^{-1}(H^{\ast}))| \\ &=|\psi(\theta(H^{\ast})\cap H^{\ast})|\\& = \alpha - 1.
		 \end{align*} 
		   Also, 
		   \begin{align*}
		   |(H_{\psi}(\theta))(H^{\ast})\cap aH| &= |\psi(\theta(H^{\ast})\cap \psi^{-1}(aH))|\\& = |\psi(\theta(H^{\ast}) \cap bH)|\\& = \gamma
		\end{align*}
		     and similarly, $|(H_{\psi}(\theta))(H^{\ast})\cap bH|=\beta$.
		     Thus, $H_{\psi}(class~\alpha \beta \gamma)$ $ \subseteq class~\alpha \gamma \beta$.
		      The converse can be proved using $\psi^{-1} \in Aut(G)$. Hence, $|\alpha \beta \gamma| = |\alpha \gamma \beta|.$
	\end{proof}
	\begin{lemma}\label{le5}
		Let G be a group and H its normal subgroup of index $3$. Suppose $\theta \in Orth(G)$ and $\psi \in Aut(G)$ such that $\psi$ preserves the cosets of $H$. Then $H_{\psi}$ will preserve the class of $Orth(G)$, that is $H_{\psi}(class~\alpha\beta\gamma) \subseteq class~\alpha\beta\gamma$. 
	\end{lemma}
	\begin{proof}
		Follows from the proof of Lemma \ref{le3}
	\end{proof}
	\begin{lemma}\label{le4}
		Let $G$ be a group and $H$ its normal subgroup of index $3$. Suppose $\theta \in Orth(G)$ and $G = H \sqcup aH \sqcup bH$. Then $\bm{|\alpha \beta 1| = \alpha |1 \alpha \beta|}$.
	\end{lemma}
	\begin{proof}
		Let $G$ be a group and $H$ its normal subgroup of index $3$. 
		Suppose $\theta \in class~1\alpha\beta$ and $h \in H^{\ast}$ such that $\theta(h) \in aH$. Then, $T_{h}[\theta](h^{-1}) = \theta(h^{-1}h)\theta(h)^{-1} \in bH$. For, $h_{1} \in H^{\ast}$ such that $\theta(h_{1}h) \in bH$, $T_{h}[\theta](h_{1}) = \theta(h_{1}h)\theta(h)^{-1} \in aH.$ There are $\beta$ such $h_{1} \in H^{\ast}$. For, $h_{2} \in H^{\ast}$ such that $\theta(h_{2}h) \in aH$, $T_{h}[\theta](h_{2}) = \theta(h_{2}h)\theta(h)^{-1} \in H^{\ast}$. From above it is clear that, $T_{h}(class~1\alpha\beta) \subseteq class~\alpha\beta1$. Since the number of $h \in H^{\ast}$ such that $\theta(h) \in aH$, in the $class~1\alpha\beta$ is $\alpha$,  $|\alpha\beta1| \geqslant \alpha|1\alpha\beta|.$
		
		Conversely, suppose $\theta \in class~\alpha\beta1$ and $h \in H^{\ast}$ such that $\theta(h) \in bH$. Then, $T_{h}[\theta](h^{-1}) = \theta(h^{-1}h)\theta(h)^{-1} \in aH.$ For, $h_{1} \in H^{\ast}$ such that $\theta(h_{1}h) \in aH$, $T_{h}[\theta](h_{1}) = \theta(h_{1}h)\theta(h)^{-1} \in bH$. There are $\beta$ such $h_{1} \in H^{\ast}$.  For, $h_{2} \in H^{\ast}$ such that $\theta(h_{2}h) \in H^{\ast}$, $T_{h}[\theta](h_{2}) = \theta(h_{2}h)\theta(h)^{-1} \in aH$. There are $\alpha-1$ such  $h_{2} \in H^{\ast}.$ Thus, $T_{h}(class~\alpha\beta1) \subseteq class~1\alpha\beta$. From above it is clear that, for $h \in H^{\ast}$, the map $T_{h}$ do not fix \textit{class}. Consequently, $T_{h_{1}}(\theta_{1}) \neq T_{h_{2}}(\theta_{2})$, where $\theta_{1}\neq\theta_{2} \in class~1\alpha\beta$ and  $h_{i} \in H^{\ast}$. If not, then $T_{h_{1}}(\theta_{1}) = T_{h_{2}}(\theta_{2})$ will implies that $T_{h_{2}^{-1}h_{1}}(\theta_{1}) = \theta_{2},$ which is a contradiction as both $\theta_{1}$ and $\theta_{2}$ belongs to the  same \textit{class}. Hence, $|\alpha \beta 1| = \alpha|1\alpha \beta|$.
	\end{proof}
	
	\section{The Alternating group $A_4$}
	Let $G = A_{4}$. Take $x,y,a,b \in A_{4}$ such that $\vert x \vert = \vert y \vert  = 2$, $\vert a \vert = 3$ and $b = a^{-1}$. Take $H = \langle x,y \rangle$ and let $D = \{d_{1}=e, d_{2}=a, d_{3}=b\}$ be ordered set  of distinct coset representatives of $H$ in $G$. Then, elements of $G$ can be ordered as $\{e,x,y,xy; a,ax,ay,axy; b,bx,by,bxy\}$. Let $\theta \in$ Orth($G$). Then, by Lemma \ref{le2}, the collapsed permutation matrix of $\theta \in$ Orth($G$) with respect to $H$ (and $D$) is ,  	$$ CM_{\theta} = \begin{pmatrix}
		\alpha&\beta&\gamma\\\beta&\gamma&\alpha\\\gamma&\alpha&\beta
	\end{pmatrix}$$ $\text{for some }\alpha,\beta,\gamma$ such that $\alpha+\beta+\gamma = 4$. Therefore,  each normalized orthomorphism of $G$ belongs to one of the classes $400, 220, 202, 211,$ $ 112, 121, 130,$ $ 103, 310, 301$.
	
	For the sake of convenience, take $x = (1~2)(3~4), y = (1~3)(2~4)$ and $a = (1~2~3)$.  Define a bijection $\sigma$ on $H$ given by $\sigma = (x,y,xy)$.
	\begin{remark}\label{re4}\hfill
		\begin{enumerate}
			\item [(a)] For $h \in H^{\ast}$, $ha = a\sigma(h)$ and $hb = b\sigma^{-1}(h).$
			\item [(b)] If two elements of a coset $aH (bH)$ differ by $h \in H^{\ast},$ then remaining two elements of the coset $aH(bH)$ also differ by same $h \in H^{\ast}.$
		\end{enumerate}
	\end{remark}
	
	\begin{lemma}\label{le6} Let $\theta \in Orth(A_{4})$. Then following holds: 
		\begin{itemize}
			\item [(i)] $\bm{\vert 130 \vert = \vert 103\vert}$,	
			\item [(ii)] $\bm{\vert 301 \vert = \vert 310\vert}$,
			\item [(iii)] $\bm{\vert 202 \vert = \vert 220\vert}$,
			\item [(iv)] $\bm{\vert 112 \vert = \vert 121\vert}$.	
		\end{itemize}
	\end{lemma}
	\begin{proof}
	 Since Aut($A_4$) is $S_4$, $i_{(2~3)} \in$ Aut($A_4$). Thus, $H_{i_{(2~3)}}$ is an homology. Also, $i_{(2~3)}(H) = H, i_{(2~3)}(aH) = bH$ and $i_{(2~3)}(bH) = aH$.  Thus proof follows from  Lemma \ref{le3}.
	\end{proof}

\begin{lemma} \label{le7}
	Let $\theta \in Orth(A_{4})$. Then following relations between classes holds:
		\begin{enumerate} 
		\item[(i)] 	$\bm{\vert 211 \vert = 2\vert 121 \vert}$,
		\item[(ii)] 	$\bm{\vert 301 \vert = 3\vert 130 \vert}$.
	\end{enumerate}
\end{lemma}
\begin{proof}
		Follows from Lemma \ref{le4}.
\end{proof}

		\begin{theorem}\label{th1}
	The number of orthomorphisms in	$class~400 = 128$.
	\end{theorem}
	\begin{proof}
		Assume $\theta \in class~400$, and $A_{ij}, 1 \leq i,j \leq 3$ are block matrices as defined in Section 2.1. Then $A_{11}$ is uniquely determined by fixing image of $x$ under $\theta$. There are two choices for $\theta(x)$, these are $y$ and $xy$. There are 8 choices for each $A_{32}$ and $A_{23}$. Hence, $|400| = 2 \times 8 \times 8 = 128.$ 
	\end{proof}

	\begin{theorem}\label{th2}
		There is no orthomorphisms in $class~220$.
	\end{theorem}
	\begin{proof}
		Consider $class~220 = s(x,y) \sqcup s(x,xy) \sqcup s(xy,x) \sqcup s(y,x) \sqcup s(xy,y) \sqcup s(y,xy)$,  where $s(x,y) = \{\theta \in \text{Orth}(G): \theta(x)=y, \theta(y) \in aH, \theta(xy) \in aH\}$.
		 Since for every $g \in G$, $i_{g} \in$ Aut($G$) such that $i_{g}$ is coset preserving. Thus, by Lemma \ref{le5}, $H_{i_{g}}(s(h_{1},h_{2})) = s(i_{g}(h_{1}), i_{g}(h_{2}))$, where $h_{i} \in H^{\ast}$. In particular, $H_{i_{a}}(s(x,y)) = s(xy,x)$,  $H_{i_{b}}(s(x,y)) = s(y,xy),$  $H_{i_{a}}(s(y,x)) = s(x,xy)$ and $H_{i_{b}}(s(y,x)) = s(xy,y)$. Also, $R(s(x,y)) = s(x,xy)$. As a result,  $\vert s(y,x) \vert = \vert s(xy,y) \vert$  $= \vert s(x,xy) \vert$ $= \vert s(x,y) \vert$ $= \vert s(xy,x) \vert$ $= \vert s(y,xy) \vert$.
		 
		  Consider $s(x,y,h_{1},h_{2}) = \{\theta \in \text{Orth}(G): \theta(x) = y, \theta(y) = ah_{1}, \theta(xy) = ah_{2}\}$, where $h_{i}  \in H$. Clearly $h_{1} \neq h_{2}$ and $\phi_{\theta}(y) \neq \phi_{\theta}(xy)$. As a result, we get $h_{1}h_{2} \neq y$. Thus, we have $4$ choices for $h_{1}$ and $2$ choices for $h_{2}$. Observe that we have $8$ subclasses of $s(x,y)$. If $\kappa$ denotes inner automorphism by an element of $H^{\ast},$ then $H_{\kappa}(s(x,y,h_{1},h_{2})) = s(x,y,a^{-1}\kappa(a)h_{1},a^{-1}\kappa(a)h_{2})$. As a result, $H_{i_{x}}(s(x,y,h_{1},h_{2})) = s(x,y,xyh_{1},xyh_{2})$, $H_{i_{y}}(s(x,y,h_{1},h_{2})) = s(x,y,xh_{1},xh_{2})$ and $H_{i_{xy}}(s(x,y,h_{1},h_{2})) = s(x,y,yh_{1},yh_{2})$. Then, under the action of group $\{H_{i_{x}}| x \in H\}$ on $\{s(x,y,h_{1},h_{2})| h_{i} \in H,h_{1} \neq h_{2}, h_{1}h_{2} \neq y)\}$ we get following two orbits of length $4$.
		  	\begin{align*}
		  	\text{ Orbit 1}  =\{&s(x,y,e,x), s(x,y,xy,y), s(x,y,y,xy), s(x,y,x,e).\}\\
		  	\text{	Orbit 2}  =	\{&s(x,y,e,xy), s(x,y,xy,e), s(x,y,y,x), s(x,y,x,y).\}
		  		  \end{align*}
	   Thus, we have $\vert s(x,y) \vert = 4\vert s(x,y,e,x) \vert + 4\vert s(x,y,e,xy) \vert$. Now we will show that $|s(x,y,e,h)| = 0$ where $h \in \{x,xy\}$. Take $\theta \in s(x,y,e,x)$. Since $\theta \in class~220,$ so $|\theta(H^{\ast}\cap aH)| = |\theta(aH) \cap H^{\ast}| = 2$. That is there exists $h_{1},h_{2},h_{3} \in H$ such that $\theta(ah_{2}) = h_{1}$ and $\theta(ah_{3}) = h_{1}y.$ Clearly, $h_{2} \neq h_{3}$ and $h_{1} \neq e ,y$. The following table shows the value of $\theta$ and $\phi_{\theta}$ for some elements of  $G$.
		\begin{table}[H]
		\begin{center}
			\begin{tabular}{|c| c| c| c| c| c| c|}
			\hline
			$z$ & $e$ & $x$ & $y$ & $xy$ & $ah_{2}$& $ah_{3}$\\\hline
			$\theta(z)$& $e$ & $y$ & $a$ & $ax$ & $h_{1}$ & $h_{1}y$ \\ \hline
			$\phi_{\theta}(z)$ & $e$ & $xy$ & $axy$& $a$&$h_{2}bh_{1}$& $h_{3}bh_{1}y$\\ \hline
		\end{tabular}\caption{}\label{t1}
		\end{center}
	\end{table}
	\vspace{-.8cm}
	From Table \ref{t1}, \[ \phi_{\theta}(ah_{2}) = h_{2}bh_{1} \text{ and } \phi_{\theta}(ah_{3}) = h_{3}bh_{1}y  \tag{$\star$}
	\]
	Clearly, $h_{2}bh_{1} \neq h_{3}bh_{1}y.$ Thus $h_{2}h_{3} \neq xy.$
	
	 Take $h_{2}h_{3} = x$. Now  by ($\star$), 
	\begin{align*}
		\phi_{\theta}(ah_{3}) =&~ xh_{2}bh_{1}y \\
		=&~ h_{2}bh_{1}\sigma^{-1}(x)y \text{  (by Remark \ref{re4})}\\
		=&~ \phi_{\theta}(ah_{2})x.
	\end{align*}
	   Since $\theta \in class~220, |\theta(bH) \cap aH| = 2.$ So, there exists $\epsilon, \eta \in bH$ such that $\theta(\eta), \theta(\epsilon) \in aH \setminus \{a, ax\}.$ Thus, by Remark \ref{re4}, $\theta(\eta)x = \theta(\epsilon)$. Also, $\phi_{\theta}(\epsilon), \phi_{\theta}(\eta) \in bH \setminus \{\phi_{\theta}(ah_{2}),\phi_{\theta}(ah_{3})\}$. Clearly, $\phi_{\theta}(\epsilon)x = \phi_{\theta}(\eta)$.
Now, from the definition of $\phi_{\theta}$, we  get $\epsilon = \eta$,  which is a contradiction.

		Take $h_{2}h_{3} = y$. Assume $\lambda$, $\tau \in aH$ such that $\theta(\lambda), \theta(\tau) \in bH$. Clearly, $\lambda, \tau \in aH \setminus\{ah_{2},ah_{3}\}$ and  $\phi_{\theta}(\lambda), \phi_{\theta}(\tau) \in aH \setminus \{\phi_{\theta}(y), \phi_{\theta}(xy)\}$. Thus, by Remark \ref{re4}, $\lambda y, = \tau$ and $\phi_{\theta}(\lambda)xy  = \phi_{\theta}(\tau).$ Now, from the definition of $\phi_{\theta}$, we get $\theta(\lambda) = \theta(\tau)$, which is a contradiction. Therefore, there is no orthomorphism in this \textit{class}, when $h=x$.
		Similarly, when $h=xy$ we will not get any orthomorphism. 
		Hence, $\vert 220 \vert = 0$.
		
	\end{proof}

	\begin{theorem}\label{th3}
		The number of orthomorphisms in	$class~130 = 32$.
	\end{theorem}
	\begin{proof}
			For $h_{1},h_{2} \in H,$ consider $r(h_{1},h_{2}) =  \{\theta \in \text{Orth}(G): \theta(x)=ah_{1}, \theta(y)$ $ = ah_{2}, \theta(xy) = ah_{3}, h_{3} \in H\}$. Then $r(h_{1},h_{2}) \subseteq class~130$. Clearly, $h_{1} \neq h_{2} \neq h_{3}$ and $\phi_{\theta}(x) \neq \phi_{\theta}(y) \neq \phi_{\theta}(xy).$ Thus, $h_{1}h_{2} \neq x, h_{1}h_{3} \neq xy$ and $h_{2}h_{3} \neq y$. Clearly, we have $4$ choices for $h_{1}$ and $2$ choices for $h_{2}$. Therefore, $class~130 = \sqcup \{r(h_{1},h_{2}) : h_{1} \neq h_{2}, h_{1}h_{2} \neq x\}$, that is, there are $8$ sub classes of $r(h_{1},h_{2})$. For $h_{1} = e$ and $h_{2} = y$, it is clear that $h_{3} = x$.
			
			  Let $\psi \in Aut(G)$ such that $\psi$ preserves the cosets of $H$. Then, $H_{\psi}(r(h_{1},h_{2})) = r(a^{-1}\psi(ah_{2}), a^{-1}\psi(ah_{3}))$. In particular,
		$H_{i_{a}}(r(e,y)) = r(x,xy)$,
		$H_{i_{b}}(r(e,y)) = r(y,e)$,
		$H_{i_{x}}(r(e,y)) = r(xy,x)$,
		 $H_{i_{a}}(r(e,xy)) = r(y,x)$, $H_{i_{b}}(r(e,xy)) = r(xy,e)$ and $H_{i_{y}}(r(e,xy)) = r(x,y)$. Also,  $R(r(e,y)) = r(y,x)$. As a result, $|r(e,y)| = |r(x,xy)| = |r(y,e)| = |r(xy,x)| = |r(e,xy)| = |r(y,x)| = |r(xy,e)| = |r(x,y)|$.

		Let $\theta \in class~130.$ By Remark \ref{re3}, $\vert \Delta_{33} \vert = 3$. Thus, $\vert \Delta_{33} \cap \theta(\Delta_{33}) \vert \geq 2$. Clearly, $\phi_{\theta}(\Delta_{33}) = H^{\ast}.$ Assume $\eta \in \Delta_{33}$ such that $\phi_{\theta}(\eta) = h_{1}, \phi_{\theta}(\theta(\eta)) = h_{2}$ and $\phi_{\theta}(\theta^{2}(\eta))=h_{1}h_{2}.$ Then, $\theta(\eta) = \eta h_{1}, \theta^{2}(\eta) = \eta h_{1}h_{2}$ and $\theta^{3}(\eta) = \eta$. Therefore, we get $ \Delta_{33} = \theta(\Delta_{33})$. Since $\Delta_{32} \sqcup \Delta_{33} = d_{3}H = \theta(\Delta_{23}) \sqcup \theta(\Delta_{33})$,  $\Delta_{32} = \theta(\Delta_{23})$. By Remark \ref{re2}, $\theta^{-1} \in class~130.$ Thus, $|\Delta_{12} \cap \theta^{-1}(\Delta_{21})| = 3$. Consider $|\Delta_{21} \cap \theta(\Delta_{12})| = |\theta^{-1}(\Delta_{21}) \cap \theta(\Delta_{12}))|= |\theta^{-1}(\Delta_{21}) \cap \Delta_{12}| = 3$. Since $\Delta_{21}\sqcup \Delta_{23} = d_{2}H = \theta(\Delta_{12}) \sqcup \theta(\Delta_{32})$,   $\Delta_{23} = \theta(\Delta_{32}).$
		
		   For $\theta \in r(e,y)$, consider Table \ref{t2}.
		\begin{table}[H]
		\begin{center}
				\begin{tabular}{|c| c| c| c|c|}
				\hline
			$z$ & $e$ & $x$ & $y$ & $xy$ \\ \hline
			$\theta(z)$ & $e$ & $a$ & $ay$ & $ax$\\ \hline
			$\phi_{\theta}(z)$ & $e$ & $ay$ & $ax$ & $a$ \\ \hline
			\end{tabular}\caption{} \label{t2}
		\end{center}
		\end{table}
		\vspace{-.7cm}
		By Table \ref{t2}, $\theta(\Delta_{12}) = \{a,ay,ax\}.$ By Remark \ref{re3}, $\theta(\Delta_{32}) = \{axy\}$. Thus, there exists  $\zeta \in \Delta_{32}$ such that $\theta(\zeta) = axy$. Since $\Delta_{23} = \theta(\Delta_{32})$ and $\Delta_{32} = \theta(\Delta_{23})$, $\theta(axy) = \zeta$. Clearly, $\phi_{\theta}(axy) \in d_{2}H \setminus \{ay,ax,a\}$. So, $\phi_{\theta}(axy) = axy$. Therefore, $\zeta = by                                                                                                                                                                                                                                                                                                                                                                                                                                                                                                                                                                                                                                                                                                                                                                                                                                                                                                                                                                                                                                                                                                                                                                                                                       $, $\Delta_{33} = \{b,bx,bxy\} = \theta(\Delta_{33})$ and $\Delta_{21} =\{a,ay,ax\}$. As a result, we get following orthomorphisms
		\begin{enumerate}
			\item [(i)] $(x,a)(y,ay)(xy,ax)(axy,by)(b,bx,bxy),$ 
			\item [(ii)] $(x,a)(y,ay)(xy,ax)(axy,by)(b,bxy,bx),$ 
			\item [(iii)]$(x,a,xy,ax,y,ay)(axy,by)(b,bx,bxy),$ 
			\item [(iv)]$(x,a,xy,ax,y,ay)(axy,by)(b,bxy,bx).$ 
		\end{enumerate}
		Hence, $|r(e,y)| = 4$ and $|130| = 32$.
	\end{proof}
	\begin{lemma}\label{le8} Let $u(h_{1},h_{2}) = \{\theta \in Orth(G): \theta(h_{1}) \in aH, \theta(h_{2}) \in aH, \theta(h_{1}h_{2}) \in bH\}$, where $h_{i} \in H^{\ast}$. Then,
		\begin{enumerate}
			\item [(i)] $H_{i_{a}}(u(h_{1},h_{2})) = u(\sigma^{-1}(h_{1}), \sigma^{-1}(h_{2}))$ and
			\item[(ii)] $H_{i_{b}}(u(h_{1},h_{2})) = u(\sigma(h_{1}), \sigma(h_{2}))$.
		\end{enumerate}
	\end{lemma}
	\begin{proof}
		The proof follows from the definition of homology.
	\end{proof}
\begin{lemma}\label{le8}\hfill
	\begin{enumerate}
		\item [(i)]  Suppose $u(x,y) = \{\theta \in Orth(G): \theta(x) \in aH, \theta(y) \in aH, \theta(xy) \in bH\}.$ Then, $\bm{\vert 121 \vert = 3\vert u(x,y) \vert}$.
		
		\item [(ii)] Suppose $u(x,y,h_{1},h_{2},h_{3}) = \{\theta \in Orth(G) : \theta(x) = ah_{1},\theta(y) = ah_{2},$ $ \theta(xy) = bh_{3}\}$. Then,	\[\bm{\vert u(x,y) \vert = 16 \vert u(x,y,e,y,e) \vert + 8 \vert u(x,y,e,y,x) \vert + 8\vert u(x,y,xy,x,e) \vert}.\]
	\end{enumerate}
\end{lemma}

\begin{proof}\begin{enumerate}
		\item [(i)] 	Clearly $class~121 = u(x,y) \sqcup u(x,xy) \sqcup u(y,xy)$, where
		$u(x,y) = \{\theta \in \text{Orth}(G): \theta(x) \in aH, \theta(y) \in aH, \theta(xy) \in bH\}$. By Lemma \ref{le8}, $H_{i_{a}}(u(x,y)) = u(x,xy)$ and  $H_{i_{b}}(u(x,y)) = u(y,xy).$ This implies $\vert u(x,y) \vert = \vert u(x,xy) \vert = \vert u(y,xy) \vert$. Therefore, 	$\vert 121 \vert = 3\vert u(x,y) \vert$.
		\item [(ii)] 	Now, $H_{i_{x}}(u(x,y)) \subseteq u(x,y)$, $R(u(x,y)) \subseteq u(x,y)$ and $RT_{xy}H_{i_{(1~4)}}$ $(u(x,y))$ $  \subseteq u(x,y)$. Clearly, the group $G_{1}$ generated by $ H_{i_{x}}, R, RT_{xy}H_{i_{(1~4)}}$ will fix $u(x,y)$ and $G_{1}$ $ \cong C_{2} \times D_{4}$. Consider $u(x,y,h_{1},h_{2},h_{3}) = \{\theta \in \text{Orth}(G) : \theta(x) = ah_{1}, \theta(y) = ah_{2}, \theta(xy) = bh_{3}\},$ where, $h_{i} \in H$. Clearly, $h_{1} \neq h_{2}$ and $\phi_{\theta}(x) \neq \phi_{\theta}(y)$, which gives $h_{1}h_{2} \neq x$. Thus, we have $4$ choices for $h_{1}$, $2$ choices for $h_{2}$ and $4$ choices for $h_{3}$. Observe that there are $32$ sub classes of $u(x,y)$ and the subgroup $G_{1}$ of the group of automorphisms of $\text{Orth}(G)$ will act naturally on $u(x,y)$. Following Table \ref{t3} shows the action of elements of $G_{1}$ on $u(x,y,h_{1},h_{2},h_{3})$.
	\end{enumerate}
		\begin{table}[H]
			\begin{tabular}{|c|c|}
					\hline
					$g \in G'$ & $g(u(x,y,h_{1},h_{2},h_{3}))$\\ \hline
					$H_{i_{x}}$ & $u(x,y,xyh_{1},xyh_{2},yh_{3})$\\ \hline
					$H_{i_{y}}$ & $u(x,y,xh_{1},xh_{2},xyh_{3})$\\ \hline
					$H_{i_{xy}}$ & $u(x,y,yh_{1},yh_{2},xh_{3})$ \\ \hline 
					$R$ & $u(x,y,yh_{1},xyh_{2},yh_{3})$ \\ \hline
					$RH_{i_{x}}$ & $u(x,y,xh_{1},h_{2},h_{3})$ \\ \hline
					$RH_{i_{y}}$ & $u(x,y,xyh_{1},yh_{2},xh_{3})$ \\ \hline
					$RH_{i_{xy}}$ & $u(x,y,h_{1},xh_{2},xyh_{3})$ \\ \hline
					$T_{xy}H_{i_{(2~3)}}$ & $u(x,y, \sigma^{-1}(i_{(2~3)}(h_{1}h_{3})),\sigma^{-1}(i_{(2~3)}(h_{2}h_{3})), \sigma^{-1}(i_{(2~3)}(h_{3})))$ \\ \hline
					$T_{xy}H_{i_{(1~4)}}$ & $u(x,y, \sigma^{-1}(xy(i_{(1~4)}(h_{1}h_{3}))),\sigma^{-1}(xy(i_{(1~4)}(h_{2}h_{3}))), \sigma^{-1}(y(i_{(1~4)}(h_{3}))))$ \\ \hline
					$T_{xy}H_{i_{(1~2~4~3)}}$ & $u(x,y, \sigma^{-1}(x(i_{(1~2~4~3)}(h_{1}h_{3}))),\sigma^{-1}(x(i_{(1~2~4~3)}(h_{2}h_{3}))), \sigma^{-1}(xy(i_{(1~2~4~3)}(h_{3}))))$ \\ \hline
					$T_{xy}H_{i_{(1~3~4~2)}}$ & $u(x,y, \sigma^{-1}(y(i_{(1~3~4~2)}(h_{1}h_{3}))),\sigma^{-1}(y(i_{(1~3~4~2)}(h_{2}h_{3}))), \sigma^{-1}(x(i_{(1~3~4~2)}(h_{3}))))$ \\ \hline
					$RT_{xy}H_{i_{(2~3)}}$ & $u(x,y, y(\sigma^{-1}(i_{(2~3)}(h_{1}h_{3}))),xy(\sigma^{-1}(i_{(2~3)}(h_{2}h_{3}))), y(\sigma^{-1}(i_{(2~3)}(h_{3}))))$ \\ \hline
					$RT_{xy}H_{i_{(1~4)}}$ & $u(x,y, \sigma^{-1}(i_{(1~4)}(h_{1}h_{3})),x(\sigma^{-1}(i_{(1~4)}(h_{2}h_{3}))), xy(\sigma^{-1}(i_{(1~4)}(h_{3}))))$ \\ \hline
					$RT_{xy}H_{i_{(1~2~4~3)}}$ & $u(x,y, x(\sigma^{-1}(i_{(1~2~4~3)}(h_{1}h_{3}))),\sigma^{-1}(i_{(1~2~4~3)}(h_{2}h_{3})), \sigma^{-1}(i_{(1~2~4~3)}(h_{3})))$ \\ \hline
					$RT_{xy}H_{i_{(1~3~4~2)}}$ & $u(x,y, xy(\sigma^{-1}(i_{(1~3~4~2)}(h_{1}h_{3}))),y(\sigma^{-1}(i_{(1~3~4~2)}(h_{2}h_{3}))), x(\sigma^{-1}(i_{(1~3~4~2)}(h_{3}))))$ \\ \hline		
				\end{tabular}\caption{} \label{t3}
			\end{table}
		\vspace{-.3cm}
		From above table, by taking values of $h_{i}$, we have following $3$ orbits, one of length $16$ and two of length $8$.
		\begin{align*}
		 			  \text{ Orbit 1}  =\{&u(x,y,e,y,e), u(x,y,xy,x,y),u(x,y,x,xy,xy),u(x,y,y,e,x),\\&u(x,y,y,x,y),u(x,y,x,y,e),u(x,y,xy,e,x),u(x,y,e,xy,xy),\\&u(x,y,e,xy,e),u(x,y,y,x,x),u(x,y,xy,e,y),u(x,y,x,y,xy),\\&u(x,y,y,e,y),u(x,y,e,y,xy),u(x,y,x,xy,e),u(x,y,xy,x,x).\}\\
		 			  \text{	Orbit 2}  =	\{&u(x,y,e,y,x), u(x,y,xy,x,xy),u(x,y,x,xy,y),u(x,y,y,e,e),\\&u(x,y,y,x,xy),u(x,y,x,y,x),u(x,y,xy,e,e),u(x,y,e,xy,y).\}\\
		 			  \text{Orbit 3} = \{&u(x,y,e,y,y), u(x,y,xy,x,e),u(x,y,x,xy,x),u(x,y,y,e,xy),\\&u(x,y,y,x,e),u(x,y,x,y,y),u(x,y,xy,e,xy),u(x,y,e,xy,x).\} 
				\end{align*}

		 Therefore, we have
		\[\vert u(x,y) \vert = 16 \vert u(x,y,e,y,e) \vert + 8 \vert u(x,y,e,y,x) \vert + 8\vert u(x,y,xy,x,e) \vert.\]
	\end{proof}
\begin{theorem}\label{th4}
     The number of orthomorphisms in $u(x,y,e,y,e)$ is $10.$
\end{theorem}
\begin{proof}
	 Consider $u(x,y,e,y,e) = \{ \theta \in Orth(G) : \theta(x) = a,\theta(y) = ay, \theta(xy) = b\}$. Clearly $\phi_{\theta}(x) =ay$, $\phi_{\theta}(y) = ax$ and $\phi_{\theta}(xy) = by$. Since $\theta \in class~121$, there exist two elements in $aH$ which map to  $H^{\ast}$ and one element in $bH$ which maps to $H^{\ast}$.  Take $h_{1} \neq h_{2}, h_{3} \in H$.  Now, there are three choices for the image of elements of $bH$ which maps to $H^{\ast}$. Based on these $3$ choices there are following three cases:   
	
	\textbf{Case 1:}  $\theta(bh_{3}) = x$. Then $\theta(ah_{1}) = y$ and $\theta(ah_{2}) = xy$. Since $\phi_{\theta}(bh_{3}) \in aH \setminus \{ay,ax\}$, $h_{3} \neq e$ and $h_{3} \neq y$. Also, $\phi_{\theta}(ah_{1}), \phi_{\theta}(ah_{2}) \in bH \setminus \{by\}$ and $\phi_{\theta}(ah_{1}) \neq \phi_{\theta}(ah_{2})$. This results in $h_{1} \neq e, h_{2} \neq y$ and $h_{1}h_{2} \neq y$. For $h_{3} = x$, $h_{1} = xy$ and $h_{2} = e$, Table \ref{t21} gives value of $\phi_{\theta}.$ Now, by deleting rows and columns which contains $1$ in $M_{\theta}$, we get reduced matrix given by Table \ref{t20}, where the entries in the cells $ij$ are $g_{j}^{-1}g_{i}$. 
	\begin{table}[H]
		\begin{center}
			\begin{tabular}{|c|c|c|c|c|c|c|}
				\hline
				$z$&$x$&$y$&$xy$&$axy$&$a$&$bx$ \\ \hline
			$\theta(z)$&$a$&$ay$&$b$&$y$&$xy$&$x$ \\ \hline
			$\phi_{\theta}(z)$&$ay$&$ax$&$by$&$b$&$bxy$&$axy$ \\ \hline
			\end{tabular}\caption{} \label{t21}
		\end{center}
	\end{table}
	\vspace{-.5cm}
		\begin{table}[H]
		\begin{center}
		\begin{tabular}{|c|c|c||c|c|c|}
			\hline
			 $\bm{\theta}$&$\bm{ax}$&$\bm{ay}$&$\bm{b}$&$\bm{by}$&$\bm{bxy}$\\ \hline
			 $\bm{ax}$&$e$&$xy$&\cellcolor{gray!25}$bx$&$b$&$bxy$\\ \hline
			 $\bm{axy}$&\cellcolor{gray!25}$y$&$x$&$bxy$&$by$&$bx$ \\ \hline \hline
			 $\bm{bx}$&$axy$&$ay$&$x$&\cellcolor{gray!25}$xy$&$y$ \\ \hline
			 $\bm{by}$&$a$&$ax$&$y$&$e$&\cellcolor{gray!25}$x$ \\ \hline
			 $\bm{bxy}$&$ax$&\cellcolor{gray!25}$a$&$xy$&$x$&$e$ \\ \hline 
		\end{tabular}\caption{Reduced Matrix with cell entries showing respective $\phi_{\theta}$} \label{t20}
		\end{center}
	\end{table}\vspace{-.7cm}

	By choosing suitable value(the chosen cell is in grey) from Table \ref{t20}, we get one orthomorphism given by $\theta_{1} = (x,a,xy,b,$ $ax,axy,y,ay,bxy,by,bx)$.	For the other values of $h_{1},h_{2}$, there is no orthomorphisms.\\
	  For $h_{3} = xy$, following the same process as above there is one orthomorphism as given in Table \ref{t4}.
	\begin{table}[H]
	\begin{center}
			\begin{tabular}{|c|c|}
			\hline
			$(h_{1},h_{2})$ & $\theta$\\ \hline
			$(y,xy)$ &  $\theta_{2} = (x,a,bxy)(y,ay)(xy,b,bx,$ $by,ax,axy)$ \\ \hline
		\end{tabular}\caption{} \label{t4}
	\end{center}
	\end{table}
	\vspace{-.7cm}
	There exists no orthomorphisms for the other values of $h_{1} \text{ and } h_{2}$.\\
	\textbf{Case 2:}  $\theta(bh_{3}) = y.$ Then $\theta(ah_{1}) = x$ and $\theta(ah_{2}) = xy$. Since $\phi_{\theta}(bh_{3}) \in aH \setminus \{ay,ax\}$, $h_{3} \neq e$ and $h_{3} \neq y$. Also, $\phi_{\theta}(ah_{1}), \phi_{\theta}(ah_{2}) \in bH\setminus \{by\}.$ and $\phi_{\theta}(ah_{1}) \neq \phi_{\theta}(ah_{2})$. This results in $h_{1} \neq x, h_{2} \neq y$ and $h_{1}h_{2} \neq xy$. For $h_{3} = x$, following the same process as in Case 1, there are two orthomorphisms as given in Table \ref{t5}.
	\begin{table}[H]
	\begin{center}
			\begin{tabular}{|c|c|}
			\hline
			$(h_{1},h_{2})$ & $\theta$\\ \hline
			$(e,x)$ & $ \theta_{3}=(x,a)(y,ay,ax,xy,b,axy,by,bxy,bx)$ \\ \hline
			$(y,xy)$ & $\theta_{4}=(x,a,axy,xy,b,by,bxy,ax,bx,y,ay)$ \\ \hline
		\end{tabular}\caption{} \label{t5}
	\end{center}
	\end{table}
	\vspace{-.7cm}
	There exists no orthomorphisms for the other values of $h_{1} \text{ and } h_{2}$.\\
	Similarly, for $h_{3} = xy$, there are two orthomorphisms as given in Table \ref{t6}. 
	\begin{table}[H]
		\begin{center}
			\begin{tabular}{|c|c|}
			\hline
			$(h_{1},h_{2})$ & $\theta$\\ \hline
			$(y,e)$ & $ \theta_{5}=(x,a,xy,b,ax,axy,bx,by,bxy,y,ay)$ \\ \hline
			$(y,xy)$ & $\theta_{6}=(x,a,ax,by,bx,bxy,y,ay)(xy,b,axy)$ \\ \hline
		\end{tabular}\caption{} \label{t6}
		\end{center}
	\end{table}
	\vspace{-.7cm}
		There exists no orthomorphisms for the other values of $h_{1} \text{ and } h_{2}$.\\
	\textbf{Case 3:}  $\theta(bh_{3}) = xy.$ Then $\theta(ah_{1}) = x$ and $\theta(ah_{2}) = y$. Since $\phi_{\theta}(bh_{3}) \in aH \setminus \{ay,ax\}$, $h_{3} \neq x$ and $h_{3} \neq xy$. Also, $\phi_{\theta}(ah_{1}), \phi_{\theta}(ah_{2}) \in bH\setminus \{by\}$ and $\phi_{\theta}(ah_{1}) \neq \phi_{\theta}(ah_{2})$. This results in $h_{1} \neq x, h_{2} \neq e$ and $h_{1}h_{2} \neq x$. For $h_{3} = e$, there is an orthomorphism as given in Table \ref{t7}.
	\begin{table}[H]
		\begin{center}
			\begin{tabular}{|c|c|}
			\hline
			$(h_{1},h_{2})$ & $\theta$\\ \hline
			$(e,xy)$ &  $\theta_{7} = (x,a)(y,ay,axy)(xy,b)(ax,by,bx,bxy)$ \\ \hline
		\end{tabular}\caption{} \label{t7}
		\end{center}
	\end{table}
	\vspace{-.7cm}
	There exists no orthomorphisms for the other values of $h_{1} \text{ and } h_{2}$.\\
	For $h_{3} = y$, there are three orthomorphisms as given in Table \ref{t8}. 
	\begin{table}[H]
		\begin{center}
			\begin{tabular}{|c|c|}
			\hline
			$(h_{1},h_{2})$ & $\theta$\\ \hline
			$(e,y)$ &  $\theta_{8} = (x,a)(y,ay)(xy,b,bxy,by)(ax,bx,axy)$ \\ \hline
			$(e,xy)$ & $\theta_{9} = (x,a)(y,ay,ax,bx,bxy,by,xy,b,axy)$\\ \hline
			$(y,x)$ & $\theta_{10} = (x,a,axy,by,xy,b,bx,bxy,ax,y,ay)$ \\ \hline
		\end{tabular}\caption{}\label{t8}
		\end{center}
	\end{table}
	\vspace{-.7cm}
	There exists no orthomorphisms for the other values of $h_{1} \text{ and } h_{2}$.\\
	Therefore, $\bm{\vert u(x,y,e,y,e) \vert = 10}$.
\end{proof}
\begin{theorem}\label{th5}
    The number of orthomorphisms in $u(x,y,e,y,x)$ is $10.$
\end{theorem}
\begin{proof}
	  Consider $u(x,y,e,y,x) = \{\theta \in \text{Orth}(G) : \theta(x) = a, \theta(y) = ay, \theta(xy) = bx\}$.  Clearly, $\phi_{\theta}(x) =ay$, $\phi_{\theta}(y) = ax$ and $\phi_{\theta}(xy) = bxy$. Since $\theta \in class~121$, there exist two element in $aH$ which maps to  $H^{\ast}$ and one element in $bH$ which maps to $H^{\ast}$.  Take $h_{1} \neq h_{2}, h_{3} \in H$.  Now, there are three choices for the image of elements of $bH$ which maps to $H^{\ast}$. Based on these $3$ choices there are following three cases:   
	
	\textbf{Case 1:}   $\theta(bh_{3}) = x.$ Then $\theta(ah_{1}) = y$ and $\theta(ah_{2}) = xy$.  Since $\phi_{\theta}(bh_{3}) \in aH \setminus \{ay,ax\}$,  $h_{3} \neq e$ and $h_{3} \neq y$. Also, $\phi_{\theta}(ah_{1}), \phi_{\theta}(ah_{2}) \in bH\setminus\{bxy\}.$ This results in $h_{1} \neq y, h_{2} \neq e$ and $h_{1}h_{2} \neq y$. For $h_{3} = x,$ following the same process as in Theorem \ref{th4} there exists an orthomorphism as given in Table \ref{t9}.
	\begin{table}[H]
		\begin{center}
			\begin{tabular}{|c|c|}
			\hline
			$(h_{1},h_{2})$ & $\theta$\\ \hline
			$(xy,y)$ &  $\theta_{1} = (x,a,ax,by,b,bxy,axy,y,ay,xy,bx)$ \\ \hline
		\end{tabular}\caption{} \label{t9}
		\end{center}
	\end{table}
	\vspace{-.7cm}
		There exists no orthomorphisms for the other values of $h_{1} \text{ and } h_{2}$.\\
	For $h_{3} = xy,$ there exists an orthomorphism as given in Table \ref{t10}.
	\begin{table}[H]
	\begin{center}
			\begin{tabular}{|c|c|}
			\hline
			$(h_{1},h_{2})$ & $\theta$\\ \hline
			$(x,y)$ &  $\theta_{2} = (x,a,ax,y,ay,xy,bx,axy,by,b,bxy)$ \\ \hline
		\end{tabular} \caption{} \label{t10}
	\end{center}
	\end{table}
	\vspace{-.7cm}
	There exists no orthomorphisms for the other values of $h_{1} \text{ and } h_{2}$.\\
	\textbf{Case 2:}   $\theta(bh_{3}) = y.$ Then $\theta(ah_{1}) = x$ and $\theta(ah_{2}) = xy$. Since $\phi_{\theta}(bh_{3}) \in aH \setminus \{ay,ax\}$,  $h_{3} \neq e$ and $h_{3} \neq y$. Also, $\phi_{\theta}(ah_{1}), \phi_{\theta}(ah_{2}) \in bH\setminus\{bxy
	\}.$ This results in $h_{1} \neq xy, h_{2} \neq e$ and $h_{1}h_{2} \neq xy$. For $h_{3} = x$, there are two orthomorphisms as given in Table\ref{t11}.
	\begin{table}[H]
	\begin{center}
			\begin{tabular}{|c|c|}
			\hline
			$(h_{1},h_{2})$ & $\theta$\\ \hline
			$(e,x)$ &  $\theta_{3} = (x,a)(y,ay,b,bxy,by,axy,ax,xy,bx)$ \\ \hline
			$(x,xy)$ & $\theta_{4} = (x,a,bxy,b,by,ax)(y,ay,axy,xy,bx)$\\ \hline
		\end{tabular}\caption{} \label{t11}
	\end{center}
	\end{table}
	\vspace{-.7cm}
			There exists no orthomorphisms for the other values of $h_{1} \text{ and } h_{2}$.\\
	For $h_{3} = xy$, there are two orthomorphisms as given in Table \ref{t12}.
	\begin{table}[H]
	\begin{center}
			\begin{tabular}{|c|c|}
			\hline
			$(h_{1},h_{2})$ & $\theta$\\ \hline
			$(x,xy)$ &  $\theta_{5} = (x,a,b,by,bxy,y,ay,ax)(xy,bx,axy)$ \\ \hline
			$(x,xy)$ & $\theta_{6} = (x,a,axy,xy,bx,b,by,ax)(y,ay,bxy)$\\ \hline
		\end{tabular} \caption{} \label{t12}
	\end{center}
	\end{table}
	\vspace{-.7cm}
			There exists no orthomorphisms for the other values of $h_{1} \text{ and } h_{2}$.\\
	\textbf{Case 3:}  $\theta(bh_{3}) = xy.$ Then $\theta(ah_{1}) = x$ and $\theta(ah_{2}) = y$.  Since $\phi_{\theta}(bh_{3}) \in aH \setminus \{ay,ax\}$, $h_{3} \neq x$ and $h_{3} \neq xy$. Also, $\phi_{\theta}(ah_{1}), \phi_{\theta}(ah_{2}) \in bH\setminus\{bxy\}.$ This results in $h_{1} \neq xy, h_{2} \neq y$ and $h_{1}h_{2} \neq x$. For $h_{3} = e$, there exists an orthomorphism as given in Table \ref{t13}.
	\begin{table}[H]
	\begin{center}
			\begin{tabular}{|c|c|}
			\hline
			$(h_{1},h_{2})$ & $\theta$\\ \hline
			$(x,xy)$ &  $\theta_{7} = (x,a,ax)(y,ay,bxy,axy)(xy,bx,by,b)$ \\ \hline
		\end{tabular} \caption{} \label{t13}
	\end{center}
	\end{table}
	\vspace{-.7cm}
		There exists no orthomorphisms for the other values of $h_{1} \text{ and } h_{2}$.\\
	For $h_{3} = y$, there are three orthomorphism as given in Table \ref{t14}.
	\begin{table}[H]
		\begin{center}
			\begin{tabular}{|c|c|}
			\hline
			$(h_{1},h_{2})$ & $\theta$\\ \hline
			$(e,xy)$ &  $\theta_{8} = (x,a)(y,ay,b,bxy,by,xy,bx,ax,axy)$ \\ \hline
			$(x,xy)$ & $\theta_{9} = (x,a,axy,y,ay,b,ax)(xy,bx,bxy,by)$\\ \hline
			$(x,xy)$ & $\theta_{10} = (x,a,bxy,axy,y,ay,ax)(xy,bx,b,by)$ \\ \hline
		\end{tabular}\caption{} \label{t14}
		\end{center}
	\end{table}
	\vspace{-.7cm}
			There exists no orthomorphisms for the other values of $h_{1} \text{ and } h_{2}$.\\
	Therefore, $\bm{\vert u(x,y,e,y,x) \vert = 10}.$
\end{proof}
\begin{theorem}\label{th6}
 	The number of orthomorphisms in  $u(x,y,xy,x,e)$ is $6.$
\end{theorem}
\begin{proof}
 Consider $ u(x,y,xy,x,e) = \{\theta \in Orth(G) : \theta(x) = axy, \theta(y) = ax, \theta(xy) = bx\}.$  Clearly, $\phi_{\theta}(x) =ax$, $\phi_{\theta}(y) = ay$ and $\phi_{\theta}(xy) = by$. Since $\theta \in class~121$, there exist two element in $aH$ which map to $H^{\ast}$ and one element in $bH$ which maps to $H^{\ast}$.  Take $h_{1} \neq h_{2}, h_{3} \in H$.  Now, there are three choices for the image of element of $bH$ which maps to $H^{\ast}$. Based on these $3$ choices there are following three cases:

\textbf{Case 1:}   $\theta(bh_{3}) = x.$ Then $\theta(ah_{1}) = y$ and $\theta(ah_{2}) = xy$. Since $\phi_{\theta}(bh_{3}) \in aH \setminus \{ay,ax\}$, $h_{3} \neq e$ and $h_{3} \neq y$. Also, $\phi_{\theta}(ah_{1}), \phi_{\theta}(ah_{2}) \in bH\setminus \{by\}.$ This results in $h_{1} \neq e, h_{2} \neq y$ and $h_{1}h_{2} \neq y$. For $h_{3} = x,$ there exists an orthomorphism as given in Table \ref{t15}.
\begin{table}[H]
\begin{center}
		\begin{tabular}{|c|c|}
		\hline
		$(h_{1},h_{2})$ & $\theta$\\ \hline
		$(y,x)$ &  $\theta_{1} = (x,axy,bx)(y,ax,xy,b,bxy,by,a,ay)$ \\ \hline
	\end{tabular}\caption{} \label{t15}
\end{center}
\end{table}
\vspace{-.7cm}
	There exists no orthomorphisms for the other values of $h_{1} \text{ and } h_{2}$.\\
For $h_{3} = xy,$ there exists an orthomorphism as given in Table \ref{t16}.
\begin{table}[H]
	\begin{center}
		\begin{tabular}{|c|c|}
		\hline
		$(h_{1},h_{2})$ & $\theta$\\ \hline
		$(y,xy)$ &  $\theta_{2} = (x,axy,xy,b,a,ay,y,ax,bx,by,bxy)$ \\ \hline
	\end{tabular} \caption{} \label{t16}
	\end{center}
\end{table}
\vspace{-.7cm}
	There exists no orthomorphisms for the other values of $h_{1} \text{ and } h_{2}$.\\
\textbf{Case 2:}   $\theta(bh_{3}) = y.$ Then $\theta(ah_{1}) = x$ and $\theta(ah_{2}) = xy$. Since $\phi_{\theta}(bh_{3}) \in aH \setminus \{ay,ax\}$, $h_{3} \neq e$ and $h_{3} \neq y$. Also, $\phi_{\theta}(ah_{1}), \phi_{\theta}(ah_{2}) \in bH\setminus \{by\}.$ This results in $h_{1} \neq x, h_{2} \neq y$ and $h_{1}h_{2} \neq xy$. For $h_{3} = x$, there exists no orthomorphism. For $h_{3} = xy$, there are two orthomorphisms as given in Table \ref{t17}.
\begin{table}[H]
\begin{center}
		\begin{tabular}{|c|c|}
		\hline
		$(h_{1},h_{2})$ & $\theta$\\ \hline
		$(e,x)$ &  $\theta_{3} = (x,axy,ay,bxy,y,ax,xy,b,by,bx,a)$ \\ \hline
		$(xy,x)$ & $\theta_{4} = (x,axy)(y,ax,xy,b,bx,by,a,ay,bxy)$ \\ \hline
	\end{tabular} \caption{} \label{t17}
\end{center}
\end{table}
\vspace{-.7cm}
	There exists no orthomorphisms for the other values of $h_{1} \text{ and } h_{2}$.\\
\textbf{Case 3:}  $\theta(bh_{3}) = xy.$ Then $\theta(ah_{1}) = x$ and $\theta(ah_{2}) = y$. Since $\phi_{\theta}(bh_{3}) \in aH \setminus \{ay,ax\}$, so we have $h_{3} \neq x$ and $h_{3} \neq xy$. Also, $\phi_{\theta}(ah_{1}), \phi_{\theta}(ah_{2}) \in bH\setminus \{by\}.$ This results in $h_{1} \neq e, h_{2} \neq x$ and $h_{1}h_{2} \neq x$. For $h_{3} = e$, there exists an orthomorphism as given in Table \ref{t18}.
\begin{table}[H]
\begin{center}
		\begin{tabular}{|c|c|}
		\hline
		$(h_{1},h_{2})$ & $\theta$\\ \hline
		$(e,xy)$ &  $\theta_{5} = (x,axy,y,ax,a)(xy,b)(bx,by,ay,bxy)$ \\ \hline
	\end{tabular}\caption{} \label{t18}
\end{center}
\end{table}
\vspace{-.7cm}
	There exists no orthomorphisms for the other values of $h_{1} \text{ and } h_{2}.$\\
For $h_{3} = y$, there exists an orthomorphisms as given in Table \ref{t19}.
\begin{table}[H]
	\begin{center}
		\begin{tabular}{|c|c|}
		\hline
		$(h_{1},h_{2})$ & $\theta$\\ \hline
		$(xy,x)$ &  $\theta_{6} = (x,axy)(y,ax)(xy,b,bx,by)(a,bxy,ay)$ \\ \hline
	\end{tabular} \caption{} \label{t19}
	\end{center}
\end{table}
\vspace{-.7cm}
	There exists no orthomorphisms for the other values of $h_{1} \text{ and } h_{2}$.\\
Therefore, $\bm{\vert u(x,y,xy,x,e) \vert = 6}$.
\end{proof}
\begin{theorem}\label{th7}
	The number of orthomorphisms in $class~121  = 864.$
\end{theorem}
\begin{proof}
 By putting the values from Theorem \ref{th4}, \ref{th5} and \ref{th6} in Lemma \ref{le8}, we get  $$\vert u(x,y) \vert = 16 \times 10 + 8 \times 10 + 8 \times 6 = 288$$ and $$\vert 121 \vert = 3 \times 288 = 864.$$
\end{proof}

 \begin{theorem}\label{th8}
 	$\vert Orth(A_{4}) \vert = 3840$.
 \end{theorem}
 \begin{proof}
 	Since  each normalized orthomorphism of Orth($A_{4}$) belongs to one of the classes $400,220,202,211,121,112,130,103,310,301$.Thus, $\vert \text{Orth}(A_{4}) \vert$ = $\vert 400 \vert + \vert 220 \vert+\vert 202 \vert+\vert 130 \vert+\vert 103 \vert+\vert 301 \vert+\vert 310 \vert+\vert 121 \vert+\vert 112 \vert+\vert 211 \vert$. From Lemma \ref{le6} and \ref{le7}, we get $\vert \text{Orth}(A_{4}) \vert = \vert 400 \vert + 2 \vert 220 \vert + 8 \vert 130 \vert + 4 \vert 121 \vert $. From Theorem \ref{th1}, \ref{th2}, \ref{th3} and \ref{th7}, we get $$\vert \text{Orth}(A_{4}) \vert = 128 + 8 \times 32 + 4 \times 864 = 3840.$$
 \end{proof}
	\section{Conclusion}
	
	In \cite{evans2}, Evans gave the idea of collapsed permutation matrix for abelian group and on that basis he partitioned the Orth($G$) into classes and finally computed number of orthomorphisms in these classes, but by using Lemma \ref{le3} and Lemma \ref{le4} we have find inter relations between these classes. That helps in reducing the calculation for the  number of orthomorphisms as we have seen in the case of Orth($A_4$). Further, using the idea of collapsed permutation matrix and Lemma \ref{le3} and Lemma \ref{le4} we can reduce the calculation in the computation of the number of orthomorphisms for any group with trivial or non-cyclic Sylow-$2$-subgroup and having index $3$ normal subgroup. For example:
	
	 \begin{enumerate}
				\item Let $\mathbb{Z}_{2} \times \mathbb{Z}_{6} = \{e,p,p^2,p^3,p^4,p^5,q,qp,qp^2,qp^3,qp^4,qp^5\}$ where, $|p|$ $ = 6$ and $|q|= 2$. Then $H =\{e, p^{3},q,qp^{3}\}$ is a index $3$ normal subgroup. By \cite{evans2}, $\text{Orth}(\mathbb{Z}_{2} \times \mathbb{Z}_{6})$ have $10$ classes. Now, $p \mapsto p^{5}$ is an automorphism which sends coset $pH$ to $p^{5}H$ and vice versa. Therefore, by Lemma \ref{le3} and \ref{le4}, we have $$|\text{Orth}(\mathbb{Z}_{2} \times \mathbb{Z}_{6})| = |400| + 2|220| + 8|130| + 4|121|.$$ One only have to find the number of orthomorphisms in $4$ classes instead of $10$.
				\item  Let $\mathbb{Z}_{15} = \langle 1 \rangle$. Then H = $\langle 3 \rangle$ is a index $3$ normal subgroup.  By \cite{evans2}, $\text{Orth}(\mathbb{Z}_{15})$ have $15$ classes. Now, $x \mapsto x^{2}$ is an automorphism which send coset $1 + H$ to $2+H$ and vice versa. Therefore, by  Lemma \ref{le3} and \ref{le4}, we have $$|\text{Orth}(\mathbb{Z}_{15})| = |500| + 10|140| + 2|230| + 2 |320| + 5|131| +  5|122|.$$ One only have to  find number of orthomorphisms in $6$ classes instead of $15.$
				
				Thus, the above method substantially reduces computational burden in the determination of orthomorphism counts for group with trivial or non-cyclic Sylow-$2$-subgroup and having index $3$ normal subgroup.
	\end{enumerate}
%
%
	
\end{document}